\documentclass[11pt,letterpaper]{amsart}

\usepackage[english]{babel}
\usepackage{amsmath,amssymb, amsthm}
\usepackage{comment}
\usepackage{tikz-cd}
\usepackage{tikz}
\usetikzlibrary{calc} 

\tikzset{
  curve/.style args={height=#1}{
    to path={%
      .. controls
        ($(\tikztostart)!0.33!(\tikztotarget)+(0,#1)$) and
        ($(\tikztostart)!0.66!(\tikztotarget)+(0,#1)$)
      .. (\tikztotarget)
      \tikztonodes
    }
  }
}

\usepackage[toc,page]{appendix}
\usepackage[hidelinks]{hyperref}
\usepackage{csquotes}

\DeclareMathOperator{\Hom}{Hom}
\DeclareMathOperator{\id}{id}

\usepackage[
backend=biber,
style=alphabetic,
sorting=nyt,
doi=true,
url=false,
maxcitenames=50,
maxnames=50,
]{biblatex}
\newcommand{\Hm}[1]{\Hom_{#1}}
 
\newcommand{\tow}[1]{\overset{#1}{\to}}

\newcommand{\T}{\mathcal{T}} 
\newcommand{\cA}{\mathcal{A}} 
 
\newcommand{\C}{\mathcal{C}} 
 
\newcommand{\cP}{\mathcal{P}} 

\newcommand{\de}{\delta}

\DeclareMathOperator{\Stab}{S}
\DeclareMathOperator{\Coh}{Coh}
\DeclareMathOperator{\Perf}{Perf}
\DeclareMathOperator{\Sing}{Sing}

\newcommand{\cS}{\mathcal{S}}

\newcommand{\Gr}{\operatorname{Gr}}

\newcommand{\Tw}{\operatorname{\Tw}}
\newcommand{\Ker}{\operatorname{Ker}}

\DeclareMathOperator{\Cone}{Cone}

\DeclareMathOperator{\Tot}{Tot}
\DeclareMathOperator{\MF}{MF}

\newtheorem{theorem}{Theorem}[section]
\newtheorem{lemma}[theorem]{Lemma}

\newtheorem{proposition}[theorem]{Proposition}
\newtheorem{corollary}[theorem]{Corollary}
\usepackage{mathtools}
\theoremstyle{definition}
\newtheorem{definition}[theorem]{Definition}

\theoremstyle{remark}
\newtheorem*{remark}{Remark}

\newtheorem{innercustomthm}{Theorem}
\newenvironment{customthm}[1]
{\renewcommand\theinnercustomthm{#1}\innercustomthm}
  {\endinnercustomthm}

\usepackage[english]{babel} 

\usepackage[hidelinks]{hyperref}
\usepackage{xr}

\title{nstepmf}
\title{An Orlov theorem for matrix factorizations with multiple factors}

\author{Alessandro Lehmann}
\address[Alessandro Lehmann]{Universiteit Antwerpen, Departement Wiskunde, Middelheimcampus,
Middelheimlaan 1,
2020 Antwerp, Belgium}
\email{alessandro.lehmann@uantwerpen.be}

\author{Nicolò Sibilla} 
\address[Nicolò Sibilla]{SISSA, Via Bonomea 265, 34136 Trieste TS, Italy}
\email{nsibilla@sissa.it}

\thanks{
The first author acknowledges financial support from the Research Foundation - Flanders (FWO) under the postdoctoral fellowship n.  1266526N
}

\makeatletter
\@namedef{subjclassname@2020}{
  \textup{2020} Mathematics Subject Classification}
\makeatother


\makeindex 

\begin{document}
\maketitle
\begin{abstract}
 We prove a generalization of Orlov's theorem \cite{orlov2004triangulated} for matrix factorizations with $n$ steps. Let $X$ be a regular scheme, $W\colon X\to \mathbb{A}^1$ a flat morphism and $D:=W^{-1}(0)$ its central fiber. We construct an appropriate triangulated category of matrix factorizations with $n$-steps \cite{tribone2021matrixfactorizationsfactors} and show that it is equivalent to the singularity category of the root stack $\sqrt[n]{(X, D)}$. We also show that this category admits a semiorthogonal decomposition into $n-1$ copies of the usual (absolute derived) category of matrix factorizations with $2$ steps.
\end{abstract}

\section{Introduction}
In this paper we study categories of $n$-step matrix factorizations. These categories have a rich structure, but are ultimately rooted in one of the most basic operations in algebra. If $R$ is a ring and $W \in R$ is a non-zero element of $R$, a matrix factorization of $W$ is a pair of square matrices $A_1$, $A_2$ with entries in $R$ satisfying 
$$
A_1  A_2 = A_2 A_1=W \cdot I
$$
where $I$ is the identity matrix. This  notion can be generalized by letting $A_1$ and $A_2$ be maps of $R$-modules (satisfying suitable properties) rather than just matrices i.e. maps of free $R$-modules. From a geometric view-point, it is natural to  further generalize this set-up and look at pairs of morphisms between coherent sheaves on a variety $X$ providing  a factorization of a regular function 
$$
W: X \to \mathbb{A}^1. 
$$ 
It turns out that matrix factorizations naturally form a  triangulated category, denoted $\MF(X,W)$. The importance of matrix factorizations within homological algebra  was first realized by Eisenbud, who in \cite{eisenbud} used them to study Cohen--Macaulay modules on hypersurface rings.  After Eisenbud's work they    became  a key tool in the algebraic study of singularities, see e.g. \cite{knorrer1987cohen}.

Remarkably,   matrix factorizations were rediscovered by Kontsevich as a mathematical encoding of the category of $B$-branes in a Landau--Ginzburg model $(X, W)$. Categories of matrix factorizations have played since a major role in homological mirror symmetry, acting as the B-side mirror of the Fukaya category in a vast number of geometric contexts, including Fano  varieties, punctured Riemann surfaces, affine hypersurfaces and more.  
Shortly after Kontsevich's proposal, Orlov proved one of the most fundamental theorems in this area, establishing a dictionary that relates  matrix factorizations and derived categories of singularities \cite{orlov2004triangulated}. 

Following Orlov, the triangulated category of singularities of a scheme $S$ is defined as the Verdier quotient of the bounded derived category of $S$ by its full  sub-category of perfect complexes 
$$
 D^b\Sing(S):=\frac{D^b\Coh(S)  }{\Perf(S) }
     $$
Now, let $X$ be a regular scheme and let  
$
W: X \to \mathbb{A}^1
$ 
be a flat morphism. Let $X_0$ be the fiber of $W$ at $0 \in \mathbb{A}^1$. Orlov proves that there is an equivalence of triangulated categories
$$
\MF(X,W) \cong  D^b\Sing(X_0)
$$

A natural generalization of the notion of matrix factorization is that of \emph{matrix factorization with multiple factors} -- that is, factorizations where instead of having two factors $A_1$ and $A_2$, one has $n$ matrices (or maps of $R$-modules) $A_1, \ldots A_n$ satisfying the conditions \[
A_1 A_2\ldots A_n=A_nA_1A_2\ldots A_{n-1}= A_2A_3\ldots A_n A_{1}=W\cdot I.
\]
While on the surface the generalization is straightforward, there are significant subtleties; namely, while the category of matrix factorization is naturally a dg-category -- a category enriched over chain complexes -- matrix factorizations with $n$ factors only have a natural enrichment over the category of $n$-complexes. This makes a homotopical study of these objects and, in particular, the definition of a (triangulated) homotopy category of $n$-step matrix factorizations a much more challenging task. The question has already been widely considered in several settings \cite{tribone2021matrixfactorizationsfactors, zbMATH07734996, frank, Zhang, Zhang2}. In particular, it has been shown that the category of $n$-matrix factorization can be given a Frobenius exact structure, giving thus a working definition of a (triangulated) homotopy category of $n$-step matrix factorizations.

Our main goal in this paper is establishing a version of Orlov's Theorem in this context; in the next section of the introduction we shall explain in greater detail our main contributions in this article.

\subsection{Structure of the paper}

We begin (Section \ref{secprel}) by recalling some relevant notions and results about exact and Frobenius categories. Then (Section \ref{secmatr}) we define the category of $n$-fold matrix factorizations with entries in an additive category and prove (Corollary \ref{nstepisfrobenius}) that it admits the structure of a Frobenius category, giving a well defined homotopy category of $n$-step matrix factorizations. At least morally, none of this is new, see e.g. \cite[Section 2]{tribone2021matrixfactorizationsfactors} or \cite[Section 4]{Zhang}. We do however streamline some of the proofs of the results that are most relevant for us. In Section \ref{secabs} we introduce the \emph{absolute derived category} of $n$-step matrix factorizations, a certain quotient of the homotopy category of matrix factorizations which generalizes the classical absolute derived category involved in the non-affine version of Orlov's theorem \cite{pomlin, Efimov_2015}; we show (Proposition \ref{affineprop}) that in the smooth affine case this reduces to the homotopy category of factorizations with projective entries. This, in particular, implies that the category considered in \cite{tribone2021matrixfactorizationsfactors, tribonebranched} coincides with the one that we study in this paper. To prove this fact, we pass through a careful study of various semiorthogonal decompositions of the relevant categories.

 The last section (Section \ref{secgeom}) of the paper contains our main result. Let $X$ be a regular scheme and let $W:X \to \mathbb{A}^1$ be a flat morphism, and let $D:=W^{-1}(0)$ be the central fiber.  On a conceptual level,  our main contribution   is highlighting the interplay between matrix factorizations and  the \emph{root stack} of the pair $(X, D)$. This is not entirely new, as constructions related to the root stack have already made an appearance in the literature on matrix factorizations (e.g. under the guise of the skew group algebra of branched covers, see   \cite{tribone2021matrixfactorizationsfactors}).  What might be new in our approach, however, is the centrality  of  
 root stacks   
 in our study of $n$-step matrix factorizations in the geometric setting.  
 Recall that the $n$-th root stack of $(X, D)$ is a DM stack equipped with a flat map to $X$
 $$
p: \sqrt[n]{(X,D)} \to X
 $$
carrying the  universal $n$-th root of the divisor 
$D$. It is classical that the abelian category coherent sheaves on $\sqrt[n]{(X,D)}$ is equivalent to the abelian category of coherent matrix factorizations for $(X,W)$: that is,  the category whose objects are length-$n$ chains of coherent sheaves  
\[
M_1\tow{\de_1}M_2\tow{}\ldots \tow{\de_{n-1}} M_n \tow{\de_n} M_1
\] such that the composition of $n$ consecutive morphisms is equal to  $W$. Our first main result is that the absolute derived category of $n$-step matrix factorizations can also  be fully described in terms of the root stack.

\begin{customthm}{A}[Proposition \ref{geommain1}]
There is an equivalence of categories 
$$
\MF^n(X, W) \cong  D^b\Sing(\sqrt[n]{(X, D)})
$$
\end{customthm}

As an easy consequence, we obtain an Orlov-type theorem for $n$-step  matrix factorizations. This is our main result. The key technical input here is given by the existence of canonical semi-orthogonal decompositions on the derived categories of root stacks, which have been studied most recently by Bodzenta--Donovan 
\cite{bodzenta2024root}.

\begin{customthm}{B}[Corollary \ref{nsteporlov}]
There is a semi-orthogonal decomposition of length $n-1$ 
$$
\MF^n(X, W) \cong  
\langle D^b\Sing(D), \ldots , 
D^b\Sing(D)
\rangle 
\cong 
\langle \MF(X, W), \ldots , 
\MF(X, W)
\rangle
$$
\end{customthm}

{\bf Acknowledgments:} This relatively short paper has had a long period of gestation. We thank  Qiangru Kuang, who was involved in the early stages of the project, for many useful discussions. We  benefited from conversation with many  colleagues on different aspects of this project. We  want to thank particularly  Sasha Efimov, Felix K\"ung, James Pascaleff, and Mattia Talpo.
\section{Preliminaries}\label{secprel}
We begin with a recollection of some basic facts about exact categories, which will also serve the purpose of fixing our notation; we refer to \cite{Exactcats} for more details.
\subsection{Exact structures}
 We will assume all of our additive categories to be idempotent complete. Let $\mathrm{E}$ be an additive category, and let $\mathcal{E}$ be a
a class of kernel-cokernel pairs in $\mathrm{E}$. We'll call sequences in $\mathcal{E}$ admissible short exact sequences, and say that a morphism is an admissible mono (resp. epi) if it appears as the first (resp. second) arrow in a sequence in $\mathcal{E}$. The class $\mathcal{E}$ is said to give an exact structure on the category $\mathrm{E}$ if it is closed under isomorphisms and satisfies the following conditions:
\begin{itemize}
    \item[E0)] For any object $M\in \mathrm{E}$, the identity morphism $\id_M$ is both an admissible mono and an admissible epi;
    \item[E1)] The classes of admissible epis and admissible monos are closed under composition;
    \item[E2)] Pushouts (resp. pullbacks) along admissible monos (resp. epis) exist and are again admissible monos (resp. epis)
\end{itemize}
An exact category is an additive category equipped with an exact structure.  We'll denote with  \[
M \hookrightarrow N \twoheadrightarrow L.
\] an admissible short exact sequence in the exact category $\mathrm{E}$; in practice, we'll often leave the class $\mathcal{E}$ implicit and call the exact structure by its underlying category.
\subsection{Pulling back exact structures}
Let $\mathrm{A}$ be an additive category, $\mathrm{E}$ an exact category, and $F\colon \mathrm{A}\to \mathrm{E}$ an additive functor. We can define a class $\mathcal{A}$ of admissible short exact sequences in $\mathrm{A}$ given by the sequences $X\hookrightarrow Y \twoheadrightarrow Z$ such that $FX\hookrightarrow FY \twoheadrightarrow FZ$  is admissible in $\mathrm{E}$.
\begin{lemma}\label{pullbackex}
    If the functor $F$ reflects and creates finite limits and colimits, then the class $\mathcal{A}$ defines an exact structure on $\mathrm{A
    }$.
\end{lemma}
\begin{proof}
    It follows immediately from the fact that $F$ creates the relevant pushouts/pullbacks that the class $\mathcal{A}$ satisfies the axioms E0-E2.
\end{proof}
\subsection{Derived categories}
If $\mathrm{E}$ is an exact category, one defines its homotopy category $K(\mathrm{E})$ as the quotient of the category of chain complexes of objects of $\mathrm{E}$ by the relation of chain homotopy; similarly one has its bounded variants $K^+(\mathrm{E})$, $K^-(\mathrm{E})$ and $K^b(\mathrm{E})$. These all come equipped with a canonical triangulated structure. A chain complex \[
\ldots \to C_{n-1}\to C_n \to C_{n+1}\to \ldots
\] is said to be acyclic if each differential splits into admissible short exact sequences
\begin{equation}\label{splitting}
\begin{tikzcd}[column sep=tiny]
	&& {Z_n} &&&& {Z_{n+1}} \\
	\ldots & {C_{n-1}} && {C_n} && {C_{n+1}} & \ldots \\
	{Z_{n-1}} &&&& {Z_{n+1}}
	\arrow[hook, from=1-3, to=2-4]
	\arrow[from=2-1, to=2-2]
	\arrow[two heads, from=2-2, to=1-3]
	\arrow[from=2-2, to=2-4]
	\arrow[from=2-4, to=2-6]
	\arrow[two heads, from=2-4, to=3-5]
	\arrow[two heads, from=2-6, to=1-7]
	\arrow[from=2-6, to=2-7]
	\arrow[hook, from=3-1, to=2-2]
	\arrow[from=3-5, to=2-6]
        \end{tikzcd}\end{equation}The derived category $D(\mathrm{E})$ is defined as the Verdier quotient of the homotopy category $K(\mathrm{E})$ by the subcategory given by the acyclic complexes; its bounded variants $D^+(\mathrm{E})$, $D^-(\mathrm{E})$ and $D^b(\mathrm{E})$ are defined similarly. In an exact category one can define projective and injective objects by imposing the relevant lifting properties against admissible epis/monos. Denote with $\operatorname{Proj}_{\mathrm{E}}$ and $\operatorname{Inj}_{\mathrm{E}}$ the full subcategories of $\mathrm{E}$ spanned by respectively the projective and injective objects. Like in the abelian case, one can show that if $\mathrm{E}$ has enough projectives the natural inclusion induces an equivalence \[K^-(\operatorname{Proj_{\mathrm{E}}})\tow{\sim} D^-(\mathrm{E}),\]
        and dually for the injective case. Again like in the abelian case, the functor $D^b(\mathrm{E})\to D^-(\mathrm{E)}$ induced by the inclusion $K^b(\mathrm{E})\to K^-(\mathrm{E)}$ is fully faithful and its essential image is given by the complexes that are eventually acyclic, in the sense that there is an $N$ for which the splitting \eqref{splitting} exists for all $n<N$.
\subsection{Frobenius categories}\label{frobeniusec}
An exact category $\mathrm{E}$ is said to be a Frobenius category if it has enough projectives and injectives, and the classes of projective and injective objects coincide. In that case, we'll denote the full subcategory spanned by this class with $\operatorname{ProjInj(\mathrm{E)}}$ and call its object the projective-injective objects; if the category is clear from the context, we'll just write $\operatorname{ProjInj}$ instead of $\operatorname{ProjInj(\mathrm{E)}}.$ Its stable category $\Stab(\mathrm{E})$ has the same objects of $\mathrm{E}$ and as hom-spaces the quotient of the ones in $\mathrm{E}$ by the ideal given by the morphism that factor through a projective-injective. By a theorem of Happel \cite{happel}, the category $\Stab(\mathrm{E})$ carries a natural triangulated structure, with shift defined by choosing, for $M\in \mathrm{E}$, an admissible short exact sequence \[
M \hookrightarrow I \twoheadrightarrow M[1] \
\] with $I$ injective. Triangles are given by the admissible short exact sequences. We'll need the following result, see \cite[Proposition 2.36]{Kalckthesis}:

\begin{proposition}\label{totfrobenius}
    For a Frobenius category $\mathrm{E}$, there is an equivalence \[
    \Stab(\mathrm{E})\cong \frac{D^b(\mathrm{E})}{K^b(\operatorname{ProjInj)}}
    \]which carries an object $X$ to the complex $X[0]$ having $X$  in degree $0$ and $0$ elsewhere.
\end{proposition}
In particular, there is a quotient functor \[
\Tot\colon D^b(\mathrm{E})\to \Stab(\mathrm{E}).
\]  which carries the complexes of the form $X[0]$ to $X\in \Stab(\mathrm{E})$. For future use, we'll need the following explicit description of the functor $\Tot$:  concretely, an element in $D^b(\mathrm{E})$ can be represented by a right-bounded complex $P$ of injective-projectives which is eventually acyclic; this means that for $n\leq N$ the complex $P$ splits into admissible short exact sequences 
\[\begin{tikzcd}[column sep=tiny]
	&& {Z_N} &&&& \\
	\ldots & {P_{N-1}} && {P_N} && {P_{N+1}} & \ldots \\
	{Z_{N-1}} &&&& {Z_{N+1}}
	\arrow[hook, from=1-3, to=2-4]
	\arrow[from=2-1, to=2-2]
	\arrow[two heads, from=2-2, to=1-3]
	\arrow[from=2-2, to=2-4]
	\arrow[from=2-4, to=2-6]
	\arrow[two heads, from=2-4, to=3-5]
	\arrow[from=2-6, to=2-7]
	\arrow[hook, from=3-1, to=2-2]
\end{tikzcd}\]  One then takes $\Tot(P)=Z_{N+1}$. An interpretation of this definition comes from the fact that  the complex $P$, truncated at $P_N$, gives a projective resolution of $\Tot(P)$. Note that different choices of truncation indices give rise to objects isomorphic in the stable category.
\subsection{Semi-orthogonal decompositions}
We will use frequently the notion of   semiorthogonal decomposition (SOD) of a triangulated category. For more information on this important construction we refer the reader to the beautiful survey \cite{kuznetsov2014semiorthogonal}. Let us fix the  notations which we  will use throughout the paper when working with SODs. Let $\mathcal{T}$ be a triangulated category, and let 
$$
\mathcal{A}_1, \ldots, \mathcal{A}_n
$$
be a finite collection of triangulated subcategories of $\mathcal{T}$. 
We denote by 
$$
\mathrm{span}\{ 
\mathcal{A}_1, \ldots, \mathcal{A}_n \}  
$$
the smallest thick triangulated subcategory of $\mathcal{T}$ containing them. We reserve the notation
$$
\big \langle 
\mathcal{A}_1, \ldots, \mathcal{A}_n \big \rangle
$$ for the case in which these subcategories make up a semiorthogonal decomposition of $\mathcal{T}$, i.e. the following conditions are satisfied
\begin{enumerate}
\item The category $\T$ is generated by the subcategories $\cA_i$, i.e. is the smallest triangulated subcategory containing all of them.
\item the subcategories $\mathcal{A}_i$ are semiorthogonal, that is for all $i < j$ and all pairs of objects $A_i \in \mathcal{A}_i$ and 
$A_j \in \mathcal{A}_j$
$$
\Hm{T}(A_j, A_i)=0.
$$
\end{enumerate}
\section{Matrix factorizations}\label{secmatr}
Let $\C$ be an additive category and $W\colon \id_\C\to \id_\C$ a natural transformation. A matrix factorization with $n$ steps of $W$ ($n$-matrix factorization for short) is given by $n$ objects $M_1, \ldots ,M_{n}$ of $\C$ together with morphisms \[
M_1\tow{\de_1}M_2\tow{}\ldots \tow{\de_{n-1}} M_n \tow{\de_n} M_1
\] such that the composition of $n$ consecutive morphisms equals the action of $W$; a morphism of $n$-matrix factorization is given by the morphisms commuting with the differentials. 
The category $\C$ always admits a trivial exact structure given by the split exact sequences. This also gives, in the same way, an exact structure on the category $\Gr_n(\C)$ of $\mathbb{Z}/n\mathbb{Z}$-graded objects. Those are trivially Frobenius, since every object is both injective and projective. There is a functor \[
\Gr_n\colon \MF^n(\C, W)\to \Gr_n(\C)
\] which sends a matrix factorization to its corresponding graded object by forgetting the morphisms $\delta_i$. The functor $\Gr_n$ satisfies the hypotheses of Lemma \ref{pullbackex}, yielding:
\begin{proposition}
    The additive category $\MF^n(\C, W)$ inherits from $\Gr_n(\C)$ an exact structure, making the functor $\Gr_n$ into an exact functor. 
\end{proposition}
We will call this  the graded-split exact structure, and denote it by $\MF^n_{ex}(\C, W)$. Explicitly, a short exact sequence \[
0\to M \to N \to L \to 0
\] is admissible if for each $i$ the sequence \[
0\to M_i \to N_i \to L_i \to 0
\]splits.
\subsection{The stable category}
The next goal is to show that the exact category $\MF^n_{ex}(\C, W)$ is a Frobenius category.
\begin{definition}
  For any $X\in \C$, we define the object $\cP_X^i\in \MF^n(\C,W)$ as \[
  X\tow{\id_X}X\tow{W}X\tow{\id_X}\ldots \tow{}X
  \] where $W$ is placed in the $i$-th position and the other arrows are all identities. It is clear that this is matrix factorization of $W$.
\end{definition}

\begin{proposition}
    The functors $\C \to \MF^n(\C, W)$\[
    X\to \cP_X^i \text{ and } X\to \cP^{i-1}_X
    \]
    are respectively left and right adjoint to the functor $\MF^n(\C, W)\to \C$ which takes the $i$-th graded piece $M_i$.
\end{proposition}
\begin{proof}
    Straightforward.
\end{proof}
\begin{proposition}\label{charprojPx}
    The objects $\cP_X[i]$ are both injective and projective in $\MF^n_{ex}(\C, W)$.
\end{proposition}
\begin{proof}
    Follows from the fact that left (resp. right) adjoints to exact functors preserve projective (resp. injective) objects, since every object is both projective and injective in the split exact structure on $\C$.
\end{proof}
In particular we obtain, for any $M\in \MF^n(\C, W)$ and $i\in \mathbb{Z}/n\mathbb{Z}$, natural (co)unit maps \[
\cP_{M_i}^i\to M \text{ and }  M \to \cP_{M_i}^{i-1}
\] given by 
\[\begin{tikzcd}
	{X_i} & {X_i} & {X_i} & \ldots & {X_{i}} & {X_i} \\
	{X_i} & {X_{i+1}} & {X_{i+2}} & \ldots & {X_{i-1}} & {X_i}
	\arrow["\id", from=1-1, to=1-2]
	\arrow["\id", from=1-1, to=2-1]
	\arrow["\id", from=1-2, to=1-3]
	\arrow["d", from=1-2, to=2-2]
	\arrow["\id", from=1-3, to=1-4]
	\arrow["{d^2}", from=1-3, to=2-3]
	\arrow["\id", from=1-4, to=1-5]
	\arrow["W", from=1-5, to=1-6]
	\arrow["d^{n-1}", from=1-5, to=2-5]
	\arrow["\id", from=1-6, to=2-6]
	\arrow["d", from=2-1, to=2-2]
	\arrow["d", from=2-2, to=2-3]
	\arrow["d", from=2-3, to=2-4]
	\arrow["d", from=2-4, to=2-5]
	\arrow["d", from=2-5, to=2-6]
\end{tikzcd}\]

and 
\[\begin{tikzcd}
	{X_i} & {X_{i+1}} & \ldots & {X_{i-2}} & {X_{i-1}} & {X_i} \\
	{X_i} & {X_i} & \ldots & {X_i} & {X_{i}} & {X_i}.
	\arrow["d", from=1-1, to=1-2]
	\arrow["\id"', from=1-1, to=2-1]
	\arrow["d", from=1-2, to=1-3]
	\arrow["d^{n-1}", from=1-2, to=2-2]
	\arrow["d", from=1-3, to=1-4]
	\arrow["d", from=1-4, to=1-5]
	\arrow["{d^2}", from=1-4, to=2-4]
	\arrow["d", from=1-5, to=1-6]
	\arrow["d", from=1-5, to=2-5]
	\arrow["\id", from=1-6, to=2-6]
	\arrow["W", from=2-1, to=2-2]
	\arrow["\id", from=2-2, to=2-3]
	\arrow["\id", from=2-3, to=2-4]
	\arrow["\id", from=2-4, to=2-5]
	\arrow["\id", from=2-5, to=2-6]
\end{tikzcd}\]
We'll call factorizations of the form $\cP_X^i$ trivial factorizations.
\begin{proposition}\label{existenceprojinj}
    The exact category $\MF^n_{ex}(\C, W)$ has enough projectives and injectives.
\end{proposition}
\begin{proof}
For any matrix factorization $M$, we can consider the maps \[
\oplus_i \cP_{M_i}^i\to M \text{ and } M \to \oplus_i \cP_{M_i}^{i-1}
\] given by the sum/product of the (co)unit maps. These admit splittings (which need not commute with the differentials) given by 
\[\begin{tikzcd}
	{\oplus_i M_i} & {\oplus_i M_i} & \ldots & {\oplus_i M_i} \\
	{M_0} & {M_1} & \ldots & {M_n}
	\arrow[from=1-1, to=1-2]
	\arrow[from=1-2, to=1-3]
	\arrow[from=1-3, to=1-4]
	\arrow["{(\id_{M_0},0,0,\ldots)}", from=2-1, to=1-1]
	\arrow[from=2-1, to=2-2]
	\arrow["{(0, \id_{M_1},0,\ldots)}"', from=2-2, to=1-2]
	\arrow[from=2-2, to=2-3]
	\arrow[from=2-3, to=2-4]
	\arrow["{(0, \ldots, \id_{M_n})}"', from=2-4, to=1-4]
\end{tikzcd}\]
and

\[\begin{tikzcd}
	{M_0} & {M_1} & \ldots & {M_n} \\
	{\oplus_i M_i} & {\oplus_i M_i} & \ldots & {\oplus_i M_i}.
	\arrow[from=1-1, to=1-2]
	\arrow[from=1-2, to=1-3]
	\arrow[from=1-3, to=1-4]
	\arrow["{\id_{M_0}}", from=2-1, to=1-1]
	\arrow[from=2-1, to=2-2]
	\arrow["{\id_{M_1}}", from=2-2, to=1-2]
	\arrow[from=2-2, to=2-3]
	\arrow[from=2-3, to=2-4]
	\arrow["{\id_{M_i}}", from=2-4, to=1-4]
\end{tikzcd}\]

\end{proof}
\begin{remark}
    In fact, the functors  \[
    \{M_i\}_{i\in \mathbb{Z}/n\mathbb{Z}}\to \oplus_i \cP^i_{M_i}\text{ and }\{M_i\}_{i\in \mathbb{Z}/n\mathbb{Z}}\to \oplus_i\cP^{i-1}_{M_i}\] are respectively left and right adjoints to the forgetful functor $\Gr_n$, and the splittings are the (co)units of the respective adjunctions.
\end{remark}
\begin{proposition}
    The classes of projective objects and injective objects in $\MF_{ex}^n(\C, W)$ coincide, and can be described as the summands of the objects of the form $\oplus_{\mathbb{Z}/n\mathbb{Z}} \cP_{X_i}$ for $X_i\in \C$.
\end{proposition}
\begin{proof}
    For any projective $Q$ one can take an injective hull $Q\to I(Q)$; since $I(Q)$ is projective, this splits and we find $I(Q)\cong Q\oplus X$. So $Q$ is a summand of an injective module, hence injective itself. The converse is dual.
\end{proof}
\begin{corollary}\label{nstepisfrobenius}
    The category $\MF_{ex}^n(\C, W)$ is a Frobenius category.
\end{corollary}
\begin{remark}
    Taking as $\C$ the category of finitely generated free modules over a regular ring and $W$ as the action of a non-zero non-unit element $f$, one recovers \cite[Theorem 2.15]{tribone2021matrixfactorizationsfactors}; we'll come back to this in Section \ref{affine}.
\end{remark}

We will thus onward denote with $\Stab(\MF^n(\C,W))$ the stable category of the exact category $\MF_{ex}^n(\C,W)$.
\subsection{The exact case}
Assume now that the category $\C$ comes already equipped with an exact structure $\mathcal{E}$. We can thus apply again Lemma \ref{pullbackex} to obtain a different exact structure on the category of matrix factorizations, this time pulled back from $\mathcal{E}$ via the functor $\Gr_n$. Since the case that we will be interested in this exact structure will be abelian, we'll denote this second exact structure with $\MF^n_{ab}(\C, W)$; note that if $\C$ is abelian, the same holds for $\MF^n_{ab}(\C, W)$.
\begin{proposition}
    If $X$ is projective (resp. injective) in $\C$, then $\cP^i_X$ is projective (resp. injective) in $\MF^n_{ab}(\C, W)$. If $\C$ has enough projectives (resp. injectives), the exact category $\MF^n_{ab}(\C, W)$ has enough projectives (resp. injectives).
\end{proposition}
\begin{proof}
    The first statement follows again by the fact that the functors $\cP^i$ are both left and right adjoint to an exact functor. For the second, let $M\in \MF^n(\C, W)$, take for every $i\in \mathbb{Z}/n\mathbb{Z}$ a projective cover\footnote{By this we simply mean that $P_i$ is a projective object surjecting onto $M_i$} $P_i \to M_i$ (resp. injective hull $M_i \to I_i$.) and consider the compositions \[
\oplus_i\cP_{P_i}^i\to \oplus_i \cP_{M_i}^i \to M \text{ and } M  \to \oplus_i \cP_{M_i}^{i-1}  \to \oplus_i \cP_{I_i}^{i-1}.
\]
\end{proof}
\begin{proposition}\label{projectiveabelian}
    A matrix factorization is projective (resp. injective) in $\MF^n_{ab}(\C, W)$ if and only if it is projective-injective in $\MF_{ex}(\C, W)$ and its graded components are projective  (resp. injective) in $\C$.
\end{proposition}

\begin{proof}
    We prove the projective case, the injective being dual. If $P$ is projective in $\MF_{ab}(\C, W)$, then it is all the more projective in $\MF_{ex}(\C, W)$. Since taking graded pieces is left adjoint to an exact functor, the graded components of $P$ must also be projective. For the converse, let $P$ be a matrix factorization that is projective in $\MF_{ex}(\C, W)$ and has projective components. Then by Proposition \ref{charprojPx} its canonical projective cover $\oplus_i \cP_{P_i}^i$ is projective in $\MF_{ab}^n(\C, W)$, and since $P$ is projective in $\MF_{ex}(\C, W)$, the map $\oplus_i \cP_{P_i}^i\to P$ must split; so $P$ is a summand of a projective object, hence projective itself.
\end{proof}

\section{The absolute derived category of matrix factorizations}\label{secabs}

Assume now that $\C$ is an abelian category. There are then two exact structures on $\MF^n(\C, W)$: the graded-split one $\MF^n_{ex}(\C,W)$, that is Frobenius, and the abelian one $\MF^n_{ab}(\C, W)$. We can thus consider the two derived categories. Explicitly, the complexes which are quotiented out to obtain $D^b(\MF^n_{ex}(\C, W))$ are those which admit a contracting homotopy that may, in general, not be a map of matrix factorizations but only of graded objects. This is a strict subclass of the larger class of acyclic (in the usual, abelian, sense) complexes. There is thus a localization functor \[
D^b(\MF^n_{ex}(\C, W))\to D^b(\MF^n_{ab}(\C, W))
\] whose kernel is the subcategory $\mathcal{K} \subseteq D^b(\MF^n_{ex}(\C, W))$ given by the acyclic  complexes. Since $\MF^n_{ex}(\C, W)$ is Frobenius, by Proposition \ref{totfrobenius} there is an equivalence \[
\frac{D^b(\MF^n_{ex}(\C, W))}{K^b(\operatorname{ProjInj)}}\tow{\sim} \Stab(\MF^n(\C, W)),
\] induced by a localization \begin{equation}\label{Toteq}
    \operatorname{Tot}\colon D^b(\MF^n_{ex}(\C, W)) \to \Stab(\MF^n(\C, W)).
\end{equation}
\begin{definition}
   We define the absolute derived category $D^{\operatorname{abs}}\MF^n(\C, W)$ of $n$-matrix factorizations as the quotient \[
D^{\operatorname{abs}}\MF^n(\C, W):=\frac{\Stab(\MF^n(\C, W))}{\operatorname{Tot(\mathcal{K})}}.
\] 
\end{definition}
By construction, we also have a commutative diagram

\[\begin{tikzcd}
	{D^b(\MF^n_{ex}(\C, W)} & {D^b(\MF^n_{ex}(\C, W)} \\
	{\Stab(\MF^n(\C, W))} & {D^{\operatorname{abs}}(\MF^n(\C, W))}
	\arrow["{q_1}", from=1-1, to=1-2]
	\arrow["\Tot"', from=1-1, to=2-1]
	\arrow["\Tot", from=1-2, to=2-2]
	\arrow["{q_2}", from=2-1, to=2-2]
\end{tikzcd}\] where the right arrow is induced by \eqref{Toteq}. Reasoning exactly as in \cite[Proposition A.3]{Efimovhfp} one sees that the right arrow is a localization with kernel given by $q_1(\Ker \Tot)=q_1(K^b(\operatorname{ProjInj}))$, which we will denote (by a slight abuse of notation) with \[
K^b(\operatorname{ProjInj})\subseteq D^b(\MF^n_{ab}(\C, W)).
\] We have thus obtained:

\begin{corollary}
    The functor $\Tot$ induces an equivalence \[
D^{\operatorname{abs}}\MF^n(\C, W)\cong \frac{D^b(\MF^n_{ab}(\C, W))}{K^b(\operatorname{ProjInj)}}
\]
\end{corollary} For future use, it's worth remarking that the subcategory $K^b(\operatorname{ProjInj})$ is generated inside $D^b(\MF^n_{ab}(\C, W))$ by the trivial matrix factorizations (not necessarily with projective entries) concentrated in degree 0.

\section{The affine case}\label{affine}
One of the classical facts about matrix factorizations (of length two) is that for smooth affine schemes, to obtain the category with the correct geometric interpretation it is not necessary to use the quotient construction of the absolute derived category; a more explicit model is simply given by the homotopy category of matrix factorization with entries that are projective (that is, vector spaces) see \cite{Efimov_2015, orlov2004triangulated, pomlin}. In this section, we show that the same holds for the $n$-fold case. This explicit model for the category of $n$-matrix factorizations with projective entries was introduced in \cite{tribone2021matrixfactorizationsfactors}. Indeed for a regular local ring $R$, taking as $\C$ the exact category of finitely generated free $R$-modules and $W$ equal to the action of a non-unit non-zerodivisor, Tribone already showed in \cite{tribone2021matrixfactorizationsfactors} that the category $\MF^n(\C, W)$ admits a Frobenius exact structure. We place ourselves in a more general setting, since our proofs still work in the non-local (and non-commutative) case.
\begin{remark}
    Since every short exact sequence of projective objects splits, the exact structure considered in \cite{tribone2021matrixfactorizationsfactors} is precisely the one pulled back from the split exact structure on $\operatorname{Proj}_\C$.
\end{remark}
Let now $\C$ be an abelian category with enough projectives and finite projective dimension. Let \[\MF^n_{\operatorname{proj}}(\C, W):=\MF^n_{ex}(\operatorname{Proj}_\C, W)\] be the Frobenius exact category of degreewise projective matrix factorizations. The inclusion $\operatorname{Proj}_\C \hookrightarrow \C$ induces a fully faithful exact functor \[
\MF^n_{\operatorname{proj}}(\C, W) \hookrightarrow \MF^n_{ex}(\C, W)
\] 

\begin{lemma} The induced functor
\[
\Stab(\MF^n_{\operatorname{proj}}(\C, W)) \hookrightarrow \Stab(\MF^n(\C, W)).
\]    is fully faithful.
\end{lemma}
\begin{proof}
    This follows from the fact that if $N\in \MF^n(\C, W)$ has projective entries, then its projective cover $P_N\to N$ also has projective entries. Hence if $M,N\in \MF^n_{\operatorname{proj}}(\C, W)$ and $f\colon M\to N$ is a morphism that factors through a projective $Q\in \MF^n_{ex}(\C, W)$, then it must also factor through the projective object $P_N$. 
\end{proof}
\begin{lemma}\label{sodabelian}
    There is a semiorthogonal decomposition \begin{equation}\label{sod1}
            K^-(\MF_{ex}^n(\C,W))=\langle K^-(\operatorname{Proj}_{\MF_{ab}^n(\C,W)}), \mathcal{K}^-\rangle,
        \end{equation}
 where we have denoted with $\mathcal{K}^-\subseteq K^-(\MF_{ex}^n(\C,W))$ the subcategory given by the acyclic (in the abelian sense) complexes.
\end{lemma}
\begin{proof}
   This follows from the standard fact that for any abelian category $\mathrm{A}$ with enough projectives, there is a SOD \[
   K^-(\mathrm{A})=\langle K^-(\operatorname{Proj}_{\mathrm{A}}), \mathcal{K}^-\rangle 
   \]owing to the fact that any right-bounded complex admits a resolution by a right-bounded complex of projectives, together with the fact that since the definition of the homotopy category does not involve the exact structure, $K^-(\MF_{ex}^n(\C,W))=K^-(\MF_{ab}^n(\C,W))$. 
\end{proof}
We will repeatedly make use of the following lemma:
\begin{lemma}\label{sodquotient}
    Let $\T$ be a triangulated category admitting a semiorthogonal decomposition \[
    \T=\langle \T_1 ,\T_2 \rangle
    \] and let $\cS \subseteq \T$ be a triangulated subcategory compatible with the SOD, in the sense that if \[S_2\to S\to S_1\to S_2[1]\] is a triangle with $S_i\in \T_i$ and $S\in \cS$, then $S_1,S_2\in \cS$. Then there is an induced semiorthogonal decomposition \[
    \T/\cS=\langle Q(\T_1), Q(\T_2)\rangle,
    \] where we have denoted with $Q$ the quotient functor $\T\to \T/\cS$.
\end{lemma}
\begin{proof}
    The nontrivial thing to show is the orthogonality in the quotient, since the existence of the necessary triangle in $\T$ descends to the quotient. Let $X_i\in \T_i$, and consider a roof  \[
    T_2\overset{s}{\leftarrow} X \tow{f} T_1,
    \] in $\T$ representing a morphism $T_2\tow{\varphi} T_1$ in the quotient $\T/\cS$, where $\operatorname{Cone}(s)\in \cS$. We have a triangle \[
    X_2\to X \to X_1 \to X_2[1]
    \] with $X_i\in \T_i$; the composition $X_2\to X \tow{f} T_1$ vanishes because of the orthogonality, so $f$ factors as 
\[\begin{tikzcd}
	{T_2} & X & {T_1} \\
	& {X_1}
	\arrow["s"', from=1-2, to=1-1]
	\arrow["f", from=1-2, to=1-3]
	\arrow[from=1-2, to=2-2]
	\arrow["{\overline{f}}"', from=2-2, to=1-3]
\end{tikzcd}\]
We also have a triangle \[
X \tow{s} T_2 \to S \to X[1]
\] with  $S\in \cS$. Applying the functor of projection to $\T_1$, we obtain a triangle \[
X_1\to 0\to S_1 \to X_1[1];
\] since by hypothesis $S_1\in \cS$, we conclude that $X_1\in \cS$ and is thus killed by the quotient. Hence the morphism $\varphi$ factors through the zero object, and we are done.
\end{proof}

The conclusion in particular follows if $\cS$ is a subcategory of one of the two factors. Hence from Lemma \ref{sodabelian} we immediately get
\begin{corollary}
    The decomposition \eqref{sod1} descends to a semiorthogonal decomposition \begin{equation}\label{sod2}
    D^-(\MF_{ex}^n(\C,W))=\langle K^-(\operatorname{Proj}(\MF_{ab}^n(\C,W)), \mathcal{K}^-\rangle,
    \end{equation}
    where by abuse of notation we have again denoted with $\mathcal{K}^-$ the subcategory of $D^-(\MF_{ex}^n(\C,W))$ given by the acyclic complexes.
\end{corollary}
Note that the composition \[
K^-(\operatorname{Proj}_{\MF_{ab}^n(\C,W)})\hookrightarrow K^-({\MF_{ab}^n(\C,W)}) \to D^-({\MF_{ab}^n(\C,W)})
\] is fully faithful by virtue of Lemma \ref{sodabelian}, identifying thus $K^-(\operatorname{Proj}_{\MF_{ab}^n(\C,W)})$ with a subcategory of $D^-({\MF_{ab}^n(\C,W)})$.

Up until this point, we have not made use of the fact that $\C$ has finite projective dimension; that will however soon be needed in order to prove that the decomposition descends to the subcategory of bounded complexes.
\begin{proposition}
    The decomposition \eqref{sod2} restricts to a semiorthogonal decomposition \[
    D^b(\MF^n_{ex}(\C, W))=\langle K^{-}(\operatorname{Proj}_{\MF_{ab}^n(\C,W)}) \cap D^b(\MF^n_{ex}(\C, W)), \mathcal{K}\rangle.
    \]
\end{proposition}

\begin{proof}
Consider the projection functor \[
P\colon  D^-(\MF_{ex}^n(\C,W))\to K^{-}(\operatorname{Proj}_{\MF_{ab}^n(\C,W)})
\] which takes a projective resolution of a complex. The claim will follow if we prove that this functor preserves the subcategory $D^b(\MF_{ex}^n(\C,W))\subseteq D^-(\MF_{ex}^n(\C,W))$. Since $D^b(\MF_{ex}^n(\C,W))$ is generated by the complexes concentrated in degree $0$, by functoriality it is enough to check the condition on those. Here the functor $P$ acts simply by taking projective resolutions (in the abelian sense), carrying an object $M\in \MF^n(\C,W)$ to a resolution\[
PM=\ldots \to P_1 \to P_0
\] where each $P_i$ is projective in $\MF^n_{ab}(\C,W)$. This resolution may well be infinite -- the category $\MF^n_{ab}(\C,W)$ does not necessarily have finite projective dimension. However by Lemma \ref{projectiveabelian} each $P_i$ has projective components, and since $\C$ (and hence a fortiori also its category of graded objects) has finite projective dimension, there must exist an index after which the sequence admits a splitting as a complex of graded objects (again, the splittings may not be morphisms of matrix factorizations). This is however precisely what it means for $PM$ to lie in $D^b(\MF_{ex}^n(\C,W))\subseteq D^-(\MF_{ex}^n(\C,W))$, and we are done.
\end{proof}

\begin{proposition}\label{affineprop}
    The semiorthogonal decomposition of $D^b(\MF_{ex}^n(\C, W))$ descends to a semiorthogonal decomposition of \[
    \frac{D^b(\MF_{ex}^n(\C, W))}{K^b(\operatorname{ProjInj})}\cong \Stab({\MF^n_{ex}(\C, W)})=\langle \Stab({\MF^n_{\operatorname{proj}}(\C, W)}), \operatorname{Tot}(\mathcal{K})\rangle.
    \] In particular, we obtain an equivalence \[
    D^{\operatorname{abs}}\MF^n(\C, W)\cong \Stab({\MF^n_{\operatorname{proj}}(\C, W)})
    \] between the absolute derived category of matrix factorizations and the stable category of factorizations with projective components considered in \cite{tribone2021matrixfactorizationsfactors}.
\end{proposition}
\begin{proof}
    There are two things to check: first that the SOD descends to the quotient, and that the subcategory given by the image of $K^{-}(\operatorname{Proj}_{\MF_{ab}^n(\C,W)})\cap D^b(\MF_{ex}(\C, W))$ via the functor \[
    \Tot\colon D^b(\MF_{ex}(\C, W)\to \Stab({\MF^n_{ex}(\C, W)})
    \] corresponds to $\Stab({\MF^n_{\operatorname{proj}}(\C, W)})$. By Lemma \ref{sodquotient}, for the first claim we just need to check that for a bounded complex of projective-injective matrix factorizations (not necessarily with entries that are projective in $\C$), its projective resolution lies again in the subcategory $K^b(\operatorname{ProjInj)}$. Since the subcategory $K^b(\operatorname{ProjInj)}$ is generated under finite direct sums, cones and direct summands by factorizations of the form $\cP^i_X$ concentrated in degree $0$, it's enough to check the claim on objects of this form.  Since $\C$ has finite projective dimension, there exists a finite projective resolution \[
    0\to Q_n \to \ldots \to Q_1\to Q_0;
    \] of $X\in \C$. Now \[
    0\to \cP^i_{Q_n} \to \ldots \to \cP^i_{Q_1}\to \cP^i_{Q_0}
    \]gives a resolution of $\cP_{X}^i$, which clearly still lies in $K^b(\operatorname{ProjInj})$. 
    
    For the second claim, we use the description of the functor $\Tot$ discussed in Section \ref{frobeniusec}. Let \[
    P=\ldots \to P_2\to P_1\to P_0 \to P_{-1}\to 0
    \] be a complex in $K^{-}(\operatorname{Proj}(\MF_{ab}^n(\C,W))\cap D^b(\MF_{ex}(\C, W))$; since $P$ is eventually acyclic (in the graded-split exact structure), after some index $N$ the complex splits giving rise in particular to a graded split short exact sequence \[
   0\to Z_N\to P_N\to \Tot(P)\to 0.
    \] Since at the graded level $\Tot(P)$ is a summand of $P_N$ and $P_N$ has projective components, the same must hold for $\Tot(P)$. Vice versa, it follows from the explicit construction of the projective covers that any resolution of a graded projective matrix factorization can be built out of graded projective modules.
\end{proof}

\section{Matrix factorizations and root stacks}\label{secgeom}
In this section we establish a link between the theory of $n$-step matrix factorizations and the singularity category of the root stack. This will allow us to deduce a version of Orlov's theorem for $n$-step matrix factorizations. 

Assume that $X$ is a smooth variety. 
Let $W:X \to \mathbb{A}^1$ be a  flat morphism. Note that by  so  called \emph{miracle flatness}, the morphism $W$ is automatically flat as soon as it is given by a regular function which is non-constant on each of the connected components of $X$. Recall that we have established a general identification 
\begin{equation}
\label{identif}
D^{\operatorname{abs}}\MF^{n}(\C, W)\cong \frac{D^b(\MF^n_{ab}(\C, W))}{K^b(\operatorname{ProjInj)}}
\end{equation}
Throughout this section $\C=\operatorname{Coh}(X)$ is the abelian category of coherent sheaves on $X$, and $W$ is given by the regular function $W:X \to \mathbb{A}^1$.  We shall adopt the notation 
$$
\MF^n(X, W):=D^{\operatorname{abs}}\MF^n(\C, W)
$$
When $n=2$ the category $\MF^2(X, W)$ coincides with the ordinary category of matrix factorizations, so we   denote  it simply  $\MF(X,W)$.

First, we need to recall some preliminaries from the theory of root stacks. We refer to \cite{talpo2018infinite} and references therein for a comprehensive introduction to these objects.   We limit ourselves to   a very short summary, geared  to the applications we have in mind. 
Let $X$ be a variety equipped with a Cartier divisor $D$. This determines a  line bundle $\mathcal{O}(D)$ and a section
$
\sigma \in \mathcal{O}(D). 
$
such that $D=\sigma^{-1}(0).$ 
\begin{definition}
The $n$-th root stack 
$
\sqrt[n]{(X,D)} 
$ 
is a DM stack equipped with a flat map 
$$p:\sqrt[n]{(X,D)} \to X$$ which carries a universal $n$-th root of the line bundle $\mathcal{O}(D)$ and of $\sigma$.  
\end{definition}
The stack $\sqrt[n]{(X,D)}$ carries a stacky divisor $D_n$ which is in a precise sense the universal $n$-th root of $D$. The divisor $D_n$ is cut out by a section $\sqrt[n]{\sigma}$ of the line bundle $\mathcal{O}(D_n),$ and there are canonical isomorphisms
$$
p^*\mathcal{O}(D) \cong \mathcal{O}(D_n)^{\otimes n} \cong \mathcal{O}(n \cdot D_n), \quad \text{and} \quad  p^*\sigma = (\sqrt[n]{\sigma})^{\otimes n}
$$

In our setting $\sigma$ is given by the regular function $W$, viewed as a section of the structure sheaf, and $D:=W^{-1}(0)$. In this context, the root stack can be easily constructed in an explicit way. Let $\sqrt[n]{\mathbb{A}^1}$ be the $n$-th root stack of the pair $(\mathbb{A}^1, \{0\})$. This can be seen to be equal to 
$$
q: \sqrt[n]{(\mathbb{A}^1, \{0\})} \cong [\mathbb{A}^1/\mu_n] \to \mathbb{A}^1
$$
where $[\mathbb{A}^1/\mu_n]$ is the quotient stack by the standard action of $n$-th roots of unity, and $p$ is   the $n$-th power map. 
Then we have that $\sqrt[n]{(X, D)}$ can be obtained as the following fiber product
\[\begin{tikzcd}
	\sqrt[n]{(X, D)} & X   \\
	{[\mathbb{A}^1/\mu_n ]} & \mathbb{A}^1
	\arrow["p", from=1-1, to=1-2]
	\arrow["W", from=1-2, to=2-2]
	\arrow["{\sqrt[n]{W}}", from=1-1, to=2-1]
	\arrow["q", from=2-1, to=2-2]
\end{tikzcd}\]
\begin{proposition}
\label{propcoh}
    There is an equivalence of abelian categories between the category of $n$-matrix factorizations $\MF_{ab}^n(\operatorname{Coh}(X), W)$ and $\operatorname{Coh}(\sqrt[n]{(X,D)})$, inducing an equivalence of derived categories \[
    D^b(\MF^n_{ab}(\Coh(X), W)) \cong D^b\Coh(\sqrt[n]{(X, D)}).
  \]
    Further, under this equivalence, the subcategory $K^b(\operatorname{ProjInj)}$ is generated, as a split-complete triangulated subcategory, by   $
p^*D^b\Coh(X)$ and its twists by the universal line bundle and its powers:  
    $$
\Big \langle p^*D^b\Coh(X), \, p^*D^b\Coh(X) \otimes \mathcal{O}(D_n), \, \ldots, \, p^*D^b\Coh(X) \otimes \mathcal{O}(D_n)^{\otimes n-1} \Big \rangle
    $$
\end{proposition}
\begin{proof}
The first statement is standard, and follows from the general theory  developed   in \cite{talpo2018infinite}. As for the second statement, it is clear once we identify $K^b(\operatorname{ProjInj)}$ with the subcategory generated by trivial factorizations. Since $p$ is flat, the pull-back $p^*$ is exact. Then $p^*D^b\Coh(X)$ can be described as the triangulated sub-category generated by trivial factorizations of the form
\begin{equation}
\label{triv}
\mathcal{F} \stackrel{=} \to \mathcal{F} 
\stackrel{=} \to \ldots \stackrel{\cdot W} \to \mathcal{F}
\end{equation}
where $\mathcal{F}$ is a coherent sheaf on $X$. The effect of tensoring this with a power of $\mathcal{O}(D_n)$ is simply to shift 
cyclically these trivial factorizations to the right of a   number of steps equal to the power of the twist we are applying. Concretely, tensoring this object with $\mathcal{O}(D_n)$ will yield the trivial factorization
$$
\mathcal{F} \stackrel{\cdot W} \to \mathcal{F} 
\stackrel{=} \to \ldots  \stackrel{=}  \to \mathcal{F}
$$
and similarly when we twist by a higher power of $\mathcal{O}(D_n)$. 
Since all projective-injective factorizations are generated by trivial factorizations of the type (\ref{triv}), and their twists, this concludes the proof.
\end{proof}

We record next the definition of singularity category. Classically, this is mostly studied for schemes, but  we shall need it in the slightly more general setting of DM stacks.

\begin{definition}
     Let $S$ be a DM stack. Then its  singularity category, denoted by $D^b\Sing(S)$, is defined as  the following localization
     $$
     \Perf(S) \hookrightarrow  D^b\Coh(S) \to  D^b\Sing(S)
     $$
\end{definition}

We will be interested in the singularity category of root stacks. Both the bounded derived category and   the category of perfect complexes of a root stack carry natural    semi-orthogonal decompositions. This will be helpful  in studying their singularity category. For the category of perfect complexes, some of the relevant references are \cite{ishii2015special}, 
\cite{bergh2016geometricity}, and 
\cite{scherotzke2020parabolic}; an explicit construction of a semi-orthogonal decomposition for the bounded derived category of coherent sheaves on a root stack was obtained in 
\cite{bodzenta2024root}.

As far as we know, the following proposition has not been explicitly stated in the literature even in the classical case $n=2$, even though in that setting it is certainly well-known to experts.
\begin{proposition}
\label{geommain1}
Assume that $X$ is a smooth variety. 
Let $W:X \to \mathbb{A}^1$ be a flat morphism, and let $D=W^{-1}(0)$. Then there is an equivalence of categories 
$$
\MF^n(X, W) \cong  D^b\Sing(\sqrt[n]{(X, D)})
$$
\end{proposition}
\begin{proof}
Combining the key equivalence (\ref{identif}) and Proposition \ref{propcoh}, it is enough to prove that the following subcategories of $D^b\Coh(\sqrt[n]{(X, D)})$ are equal: 
\begin{enumerate}
\item The first is the subcategory generated by $p^*D^b\Coh(X)$ and its twists, that is 
$$
\mathrm{span}\{ p^*D^b\Coh(X), \, p^*D^b\Coh(X) \otimes \mathcal{O}(D_n), \, \ldots, \, p^*D^b\Coh(X) \otimes \mathcal{O}(D_n)^{\otimes n-1} \}
    $$
    \item The second is the category of perfect complexes on $\sqrt[n]{(X, D)}$, 
    $$
    \Perf(\sqrt[n]{(X, D)})
    $$
\end{enumerate}
Now note that there is an obvious inclusion of category $(1)$ inside category $(2)$. Indeed by assumption $X$ is smooth, hence there is an identification 
$$
\Perf(X) \cong D^b\Coh(X).
$$
Further perfect complexes are stable under pull-backs, and the tensor product of perfect complexes is again a perfect complex. So all the generators of category $(1)$ are contained in $
    \Perf(\sqrt[n]{(X, D)})
    $; and hence  also the subcategory they generate is contained in $
    \Perf(\sqrt[n]{(X, D)})
    $.  

    Let us now prove the opposite inclusion. The categories involved in our statement have   good local-to-global behaviour: in particular, they satisfy descent with respect to the pull-back of Zariski open covers along the root map 
    $$
    p: \sqrt[n]{(X, D)} \to X 
    $$
This allows us to reduce to the case when $X$ is affine.  

It follows from the semi-orthogonal decompositions constructed in \cite{ishii2015special, bergh2016geometricity}  that $\Perf(\sqrt[n]{(X, D)})$ is generated by the following subcategories: 
\begin{itemize}
\item The image of the category of perfect complexes on $X$ along $p^*$, $p^*(\Perf(X))$
\item Full subcategories $\Perf(D_n)_\chi$, which are indexed by characters  $\chi$ of $\mu_n$.  
 Each of these is equivalent to the category of perfect complexes on $D$,
$$
\Perf(D_n)_\chi \cong \Perf(D)
$$
\end{itemize}
The reader can find    the definition of $\Perf(D_n)_\chi$ in  \cite{bergh2016geometricity}. We will not recall it here  as in fact we do not need it. Indeed, since  $X$ is affine, we can express these  subcategories  in a more direct way. The divisor $D$ is  affine, and hence its category of perfect complexes is  generated by the structure sheaf
$$
\Perf(D) \cong \mathrm{span}\{\mathcal{O}_D \}.
$$
So instead of looking at the subcategories $\Perf(D_n)_\chi$, we can just consider as generators the structure sheaf of the stacky divisor $D_n$ on $\sqrt[n]{(X, D)}$,  and its twists by $\mathcal{O}(D_n)$ and its powers. This  implies that in order to prove that $
    \Perf(\sqrt[n]{(X, D)})
    $ is contained in category $(1)$, it is sufficient to show that the sheaf  $\mathcal{O}_{D_n}$ lies in category $(1)$. Indeed, then its twists would also  be automatically contained in category $(1)$.

We can reduce to the  case when the pair $(X,D)$ is given by $(\mathbb{A}^1, 0)$, where we denote by   
$
q: \sqrt[n]{\mathbb{A}^1} \to \mathbb{A}^1
$ 
  the root map. This follows from flat base-change applied to the fiber product 
\[\begin{tikzcd}
	D_n & \sqrt[n]{(X, D)}   \\
	 B \mu_n & {[\mathbb{A}^1/\mu_n ].}
	\arrow["j", from=1-1, to=1-2]
	\arrow["{\sqrt[n]{W}}", from=1-2, to=2-2]
	\arrow["l", from=1-1, to=2-1]
	\arrow["i", from=2-1, to=2-2]
\end{tikzcd}\]
In the diagram above $B \mu_n$ is the $n$-th root of the divisor $\{0\}$ inside $$[\mathbb{A}^1/\mu_n ] \cong \sqrt[n]{\mathbb{A}^1}$$  Note that the map $\sqrt[n]{W}$ is flat because it is obtained by base-change from the regular map $W$, which is flat. Flat base change yields an  equivalence 
\begin{equation}
\label{flatbasechange}
\mathcal{O}_{D_n} \cong j_*(l^*\mathcal{O}_{B\mu_n}) \cong \sqrt[n]{W}^*(i_* \mathcal{O}_{B\mu_n}).
\end{equation}
 Thanks to (\ref{flatbasechange}), in order to prove the claim it is enough  to show that $i_* \mathcal{O}_{B\mu_n}$ lies in the subcategory of $D^b\Coh(\sqrt[n]{\mathbb{A}^1})$ generated by 
 $q^*(D^b\Coh( \mathbb{A}^1))$, and its twists.  
 Note that indeed the subcategory 
$$
\operatorname{span}\{ q^*D^b\Coh(\mathbb{A}^1),  
\, \ldots, \, q^*D^b\Coh(\mathbb{A}^1) \otimes \mathcal{O}(\{0\}_n)^{\otimes n-1} \} \subset D^b\Coh(\sqrt[n]{\mathbb{A}^1})
$$
 pulls back along $\sqrt[n]{W}^*$ to the subcategory $(1)$ of $D^b\Coh(\sqrt[n]{(X, D)} )$. 
  
So we are reduced to the case of the pair  $(\mathbb{A}^1,0)$. This can be proved via a simple direct calculation. Note that since $\sqrt[n]{\mathbb{A}^1}$ is isomorphic to the global quotient $[\mathbb{A}^1/\mu_n ]$, we can view coherent sheaves on this stack as $\mu_n$-equivariant sheaves on $\mathbb{A}^1,$ i.e. as $\mu_n$-equivariant $k[T]$-modules. The universal line bundle on $\sqrt[n]{\mathbb{A}^1}$ corresponds to $k[T]$ equipped with the $\mu_n$-equivariant structure such that $k$ has weight $1$. 
 Let $k(0)$ be the skyscraper sheaf at $0 \in \mathbb{A}^1.$ Then $q^*(k(0))$ is given by the equivariant $K[T]$-module 
 $k[T]/T^n$  such that $k$ has weight $0$ (and hence $T$ has weight $1$). It is convenient to keep track of  the equivariant structure, so we will denote this object by
 $$
 (k[T]/T^n)_0
 $$
 Similarly we shall denote 
 $$
 (k[T]/T^n)_m
 $$
 the same object but with the equivariant structure which places $k$ in weight $m$.

 Now consider the following four term exact sequence of equivariant $K[T]$-modules
 \begin{equation}
 \label{4terms}
 0 \to k_0 \to (k[T]/T^n)_1 \stackrel{- \cdot T }\to (k[T]/T^n)_0 \to k_0 \to 0
\end{equation}
The equivariant $K[T]$-module $k_0$, i.e. $k$ equipped with the trivial $\mu_n$-action,  corresponds to the structure sheaf of the stacky divisor $\{0\}_n$ on $\sqrt[n]{\mathbb{A}^1}$. Note that $k_0$ has no non-trivial 
self-extensions. Hence, sequence (\ref{4terms}) implies that we have an isomorphism in $D^b\Coh(\sqrt[n]{\mathbb{A}^1})$
$$
\Cone \Big ((k[T]/T^n)_1 \stackrel{- \cdot T }\to (k[T]/T^n)_0) \Big ) \cong k_0 \oplus k_0[1]
$$
We see therefore that $k_0$ can be expressed as a summand of the cone of a map between an object in $q^*D^b\Coh(\mathbb{A}^1)$ and an object in $q^*D^b\Coh(\mathbb{A}^1) \otimes \mathcal{O}(\{0\}_n)$. In particular, it is contained in the subcategory
$$
\mathrm{span}\{ q^*D^b\Coh(\mathbb{A}^1),  
\, \ldots, \, q^*D^b\Coh(\mathbb{A}^1) \otimes \mathcal{O}(\{0\}_n)^{\otimes n-1} \} \subset D^b\Coh(\sqrt[n]{\mathbb{A}^1})
$$
and this concludes the proof.
\end{proof}
We place ourselves under the assumptions of Proposition \ref{geommain1}. Recall that  $D$  is the central fiber $W^{-1}(0)$.  One of the key  results in the theory of matrix factorizations is a theorem of Orlov \cite{orlov2004triangulated}  stating that there is an equivalence 
$$
\MF(X, W) \cong D^b\Sing(D)
$$
The next result gives an analogue of Orlov's theorem in the setting of $n$-fold matrix factorizations.
\begin{corollary}
\label{nsteporlov}
There is a semi-orthogonal decomposition of length $n-1$ 
$$
\MF^n(X, W) \cong  
\langle D^b\Sing(D), \ldots , 
D^b\Sing(D)
\rangle 
\cong 
\langle \MF(X, W), \ldots , 
\MF(X, W)
\rangle
$$
\end{corollary}
\begin{proof}
This follows easily from Proposition \ref{geommain1} and \cite{bodzenta2024root}. In particular, Theorem 1.2 from that paper gives a semi-orthogonal decomposition of $D^b\Coh(\sqrt[n]{(X, D)})$ given by 
$$
\Big \langle  D^b\Coh(D) \otimes \mathcal{O}((1-n)D_n),  \, \ldots, \,  D^b\Coh(D) \otimes \mathcal{O}(-D_n) , p^*D^b\Coh(X) \Big \rangle
    $$
As the authors explain in the comments right after the statement of Theorem 1.2, this extends to $D^b\Coh(\sqrt[n]{(X, D)})$  the semi-orthogonal decomposition on the category of perfect complexes of the root stack which had been constructed in \cite{bergh2016geometricity}. Indeed, as we have already mentioned, in \cite{bergh2016geometricity} the authors construct a semi-orthogonal decomposition for $\Perf(\sqrt[n]{(X, D)})$ with factors
$$
\Big \langle  \Perf(D) \otimes \mathcal{O}((1-n)D_n),  \, \ldots, \,  \Perf(D) \otimes \mathcal{O}(-D_n) , p^*D^b\Coh(X) \Big \rangle
    $$
Further, as explained in  \cite{bodzenta2024root}, 
the natural inclusion 
    $$
\Perf(\sqrt[n]{(X, D)}) \subset D^b\Coh(\sqrt[n]{(X, D)})
    $$
    is compatible with the semi-orthogonal decompositions on its source and on its target.  
That is, we are in the situation considered in Lemma \ref{sodquotient}. Note that we had stated the Lemma only for length two semi-orthogonal decompositions, but the conclusion holds for SODs with an arbitrary number of factors.

Using Proposition \ref{geommain1} we conclude that 
$$
\MF^n(X, W) \cong  D^b\Sing(\sqrt[n]{(X, D)}) = D^b\Coh(\sqrt[n]{(X, D)}) / \Perf(\sqrt[n]{(X, D)})
$$
admits a length $n-1$ semi-orthogonal decomposition with factors
$$
D^b\Coh(D)/\Perf(D) = D^b\Sing(D)
$$
 and this concludes the proof.
\end{proof}
\printbibliography
\end{document}